\documentclass{birkjour}

\usepackage{color}
 \newtheorem{thm}{Theorem}[section]
 \newtheorem{cor}[thm]{Corollary}
 \newtheorem{lem}[thm]{Lemma}
 
 \theoremstyle{definition}
 
 \theoremstyle{remark}
 
 \newtheorem*{ex}{Example}
 \numberwithin{equation}{section}

\newcommand{\p}{\varphi}
\newcommand{\mS}{\mathcal S}

\newcommand{\Z}{\mathbb Z}
\newcommand{\N}{\mathbb N}
\newcommand{\bR}{\mathbb R}
\newcommand{\R}{\mathbb R}
\newcommand{\mpS}{\mathcal{S}_{\on{mp}}}
\newcommand{\IMS}{M_\mathcal{S}}

\newcommand{\ve}{\varepsilon}
\newcommand{\on}{\operatorname}
\newcommand{\vp}{\varphi}

\begin{document}

%
%
%
%
%
%
%
%
%

\title[Contractivity of idempotent Markov operators]
 {Existence of invariant idempotent measures by contractivity of idempotent Markov operators}


\author[Rudnei D. da Cunha]{Rudnei D. da Cunha}
\address{Instituto de Matem\'atica e Estat\'istica - UFRGS, Av. Bento Gon\c calves 9500, 91500-900 Porto Alegre - RS - Brazil}
\email{rudnei.cunha@ufrgs.br}

\author[Elismar R. Oliveira]{Elismar R. Oliveira}
\address{Instituto de Matem\'atica e Estat\'istica - UFRGS, Av. Bento Gon\c calves 9500, 91500-900 Porto Alegre - RS - Brazil}
\email{elismar.oliveira@ufrgs.br}

\author[Filip Strobin]{Filip Strobin}
\address{%
Lodz University of Technology, W\'olcza\'nska 215, 90-924 {\L}\'od\'z, Poland}
\email{filip.strobin@p.lodz.pl}


\date{September 23, 2021}

\begin{abstract}
We prove that the idempotent Markov operator generated by contractive max plus normalized iterated function system (IFS) is also a contractive map w.r.t. natural metrics on the space of idempotent measures. This gives alternative proofs of the existence of invariant idempotent measures for such IFSs.
\end{abstract}

\maketitle
\section{Introduction}
$\;$\\
The idempotent analysis, that was introduced by  Maslov and his collaborators \cite{Lit02} and \cite{Lit07}, brought the notion of {an} idempotent (or Maslov) measure
with important applications in many areas of mathematics, such as optimization, mathematical morphology and game theory. It can be considered as a non additive integration theory built over a max-plus semiring.

The natural question of the existence of an idempotent version of the Hutchinson-Barnsley {theory} was considered recently by Mazurenko and Za\-ri\-chnyi in \cite{MZ} and by the authors in \cite{COS}. It turned out {that it is possible to associate in a {``}reasonable way" an invariant idempotent measure to each contractive IFS $\mS$}. In a natural way there can be defined a counterpart of  the Markov operator $M_\mS$ acting on the space of idempotent measures $I(X)$ (where $X$ is a complete or compact metric space), and the invariant idempotent measure is exactly the contractive fixed point (w.r.t. the canonical {pointwise convergence} topology $\tau_p$ on $I(X)$) of it.

The proof presented in \cite{MZ} is topological and does not base on the possible contractiveness of $\IMS$ w.r.t. some metric on $I(X)$. In \cite{COS} we provided an alternative proof by defining a certain metric $d_\theta$, induced by an embedding of the space of idempotent measures into the space of fuzzy sets. The metric $d_\theta$ induces topology stronger than $\tau_p$ and such that $\IMS$ is contractive w.r.t. it. Thus its contractive fixed point is exactly the invariant idempotent measure.

In the present paper, we show that the operator $\IMS$ is contractive w.r.t. other natural metrics on the space $I(X)$ - the one defined by Zaitov in \cite{ZAI}, and a version of the one considered by Bazylevych, Repov\v{s} and Zarichnyi in \cite{Rep10}. This gives alternative proofs of the existence of invariant idempotent measures for such IFSs.

\section{Preliminaries}
In this section we will give basic definitions and facts concerning idempotent measures and iterated function systems (IFSs for short). Our notation is based on that given in our paper \cite{COS}, where much more details are given - here we present the material in a short way. We also refer the reader to papers \cite{Kol97}, \cite{Aki99}, \cite{MoDo99}, \cite{Kol88}, \cite{ZAI}, \cite{Rep10} or \cite{MZ} for brief expositions.

By the \emph{max-plus semiring} we will mean the triple $S=(\bR_{\max},\oplus,\odot)$, where {$\mathbb{R}_{{\max}}:=\mathbb{R} \cup \{-\infty\}$} and $x\oplus y:=\max\{x,y\}$ and $x\odot y:=x+y$ for $x,y\in\bR_{\max}$.

For a compact metric space $X$, by $C(X)$ we consider the space of continuous maps from $X$ to $\bR$.

A functional (not necessarily linear nor continuous) $\mu:C(X) \to \mathbb{R}$ satisfying
  \begin{enumerate}
    \item $\mu(\lambda)=\lambda$ for all $\lambda\in\R$ (normalization);
    \item $\mu(\lambda \odot \psi)=\lambda \odot \mu(\psi)$, for all $\lambda \in\mathbb{R}$ and $\psi \in C(X)$;
    \item $\mu(\varphi \oplus \psi)=\mu(\varphi) \oplus \mu(\psi)$, for all $\varphi, \psi \in C(X)$,
  \end{enumerate}
  is called an \emph{idempotent probability measure} (or \emph{Maslov measure}).

 By $I(X)$ we denote the family of all idempotent probability measures. Canonically, we endow $I(X)$ with the  pointwise convergence topology $\tau_p$, whose subbase consists of sets $V(\mu,\p,\ve):=\{\nu\in I(X):|\nu(\p)-\mu(\p)|<\ve\}$, where $\mu\in I(X),\;\p\in C(X)$  and $\ve>0$. Note that $I(X)$ is compact w.r.t. $\tau_p$ provided $X$ is compact (see for example \cite[Theorem 5.3]{Rep10}).

By the \emph{density} of an idempotent probability measure $\mu$ we will mean the unique upper semicontinuous (usc) map $\lambda_\mu:X\to\R_{\max}$ such that
 $\lambda_{\mu}(x)=0$ for some $x\in X$ and
$
\mu=\bigoplus_{x\in X}\lambda_{\mu}(x)\odot\delta_x,$ that is, for every $\vp\in C(X)$, we have
\begin{equation}\label{fil1}
\mu(\vp)=\bigoplus_{x\in X}\lambda_{\mu}(x)\odot\vp(x)=\max\{\lambda_{\mu}(x)+\vp(x):x\in X\}{.}
\end{equation}
Note that, conversely, each usc map $\lambda:X\to[-\infty,0]$ with $\lambda(x)=0$ for some $x\in X$, is the density of some idempotent measure.

An important notion is the support of an idempotent measure (see, e.g., \cite{ZAI}, \cite{Zar}). We give here an equivalent formulation: For $\mu\in I(X)$, we set
$
\operatorname{supp} \mu:=\overline{\left\{x \in X: \lambda_\mu(x)>-\infty\right\}}.
$

{For another metric space $Y$ and a continuous map $\phi: X \to Y$, define the \emph{max-plus pushforward map} $I(\phi): I(X) \to I(Y)$ by
$$I(\phi)(\mu)(\varphi):=\mu( \varphi\circ \phi), \; \forall \varphi \in C(Y),$$
for any $\mu \in I(X)$.}

Finally, we are ready to define \emph{max-plus normalized} IFSs and invariant idempotent measures. For brief expositions, see \cite{BAR88}, \cite{HUT} and \cite{MZ}.

{By a } \emph{max-plus normalized IFS} {we will mean any triple $$\mpS=(X,(\phi_j)_{j=1}^L,(q_j)_{j=1}^L)$$ such that $(X,d)$ is a complete metric space, $\phi_j$, $j=1,...,L$, are continuous selfmaps of $X$ and $(q_j)_{j=1}^L$ is a family of real numbers so that $\max\{q_j:j=1,...,L\}=0$.

The map $M_{\mathcal{S}}:I(X)\to I(X)$, defined by
\begin{equation*}
\forall_{\mu\in I(X)}\;M_{\mathcal{S}}(\mu):=\bigoplus_{j=1}^Lq_j\odot (I(\phi_j)(\mu))
\end{equation*}
will be called }\emph{the idempotent Markov operator} {generated by $\mS$.

By  }the \emph{invariant idempotent measure} {of {$\mpS$} we mean the unique measure $\mu_{\mathcal{S}}\in I(X)$ which is a fixed point of $M_\mS$ and for every $\mu\in I(X)$, the sequence of iterates $M^{(n)}_{\mathcal{S}}(\mu)$ converges to $\mu_{\mathcal{S}}$ with respect to the topology $\tau_p$ on $I(X)$.

Finally, we say that $\mpS$ is:\\
- }\emph{Banach contractive},{ if the Lipschitz constants $\on{Lip}(\phi_j)<1$ for $j=1,...,L$.\\
- }\emph{Matkowski contractive},{ if each map $\phi_j$ is a Matkowski contraction, that is, there exists a nondecreasing map $\varphi_j:[0, \infty)\to[0,\infty)$ (called as a witness for $\phi_j$) such that $\lim_{n\to\infty}\varphi_j^{(n)}(t)= 0$ for any $t>0$, and
\begin{equation}\label{1}
\forall_{x,y\in X}\;d(\phi_j(x),\phi_j(y))\leq\varphi_j(d(x,y)).
\end{equation}}
Note that for compact space $X$, the map $\phi:X\to X$ is a Matkowski contraction if, and only if, $d(\phi(x),\phi(y))<d(x,y)$ for $x\neq y$. Despite the fact that in main results we are interested in compact spaces, we will use condition (\ref{1}) since it allows to make use of concrete witness in computations. We refer the reader to \cite{JJ07} for a survey on different contractive conditions.

As mentioned in the introduction, recently it has been proved that each Banach (in \cite{MZ}) or Matkowski (in \cite{COS}) contractive max-plus normalized IFS $\mpS$ on a compact space generates the invariant idempotent measure. The proof in \cite{MZ}  is rather topological and does not involve contractivity of $M_\mS$ w.r.t. some metric. On the other hand, in \cite{COS} we gave an alternative proof by defining a certain metric $d_\theta$ on $I(X)$ and showing that $M_\mS$ is contractive w.r.t. this metric.  The main results of our paper show that $M_\mS$ is Matkowski or Banach contractive w.r.t. other natural metrics on $I(X)$.\\
Finally, let us note that in \cite{MZ} the existence of the invariant idempotent measure has been established for contractive max-plus normalized IFSs on complete spaces. However, the proof presented there is given for compact spaces, and then the result is lifted to all complete spaces by standard properties of contractive IFSs. Moreover, in \cite{MZ2} the existence of invariant idempotent measures is established via contractivity of $M_{\mS}$ w.r.t. a certain metric, but only for ultrametric spaces.

\section{Contractivity w.r.t. Zaitov's metric $d_1$}
Throughout the rest of this section, we assume that $(X,d)$ is a compact metric space.
As mentioned earlier, in \cite{ZAI} there was defined a metric $d_I$ on $I(X)$ that generates the topology $\tau_p$. Its definition  is complicated, but in a natural way there can be defined a metric $d_1$ on $I(X)$ so that $d_I\leq d_1$ and in particular, the topology induced by $d_1$ is finer than $\tau_p$. The metric $d_1$ is defined as follows: for $\mu_1,\mu_2\in I(X)$, set
$$
d_{1}\left(\mu_{1}, \mu_{2}\right):=\inf \left\{\sup \{d(x, y):(x, y) \in \operatorname{supp} \xi\}: \xi \in \Lambda_{\mu_1,\mu_2}\right\}
$$ where
$\Lambda_{\mu_1,\mu_2}$ is the family of all idempotent measures $\xi\in I(X\times X)$ with $I\left(\pi_{i}\right)(\xi)=\mu_{i}, i=1,2$, and where $\pi_i$, $i=1,2$, are natural projections of $X\times X$ onto $X$.

We start with a technical lemma that will be useful later on.
\begin{lem}\label{newlemma1}
Let $(\mu_n),(\nu_n)\subset I(X)$ be $\tau_p$-convergent sequences and $\mu,\nu\in I(X)$ be their limits. Let $(\xi_n)\subset I(X^2)$ be such that $\xi_n\in\Lambda_{\mu_n,\nu_n}$ for every $n\in\N$. Then there exists $\xi\in \Lambda_{\mu,\nu}$ such that
$$
\sup\{d(x,y):(x,y)\in\on{supp}(\xi)\}\leq \liminf_{n\to\infty}\sup\{d(x,y):(x,y)\in\on{supp}(\xi_n)\}.
$$
\end{lem}
\begin{proof}
First take a subsequence $(\xi_{k_n^{(1)}})$ so that
$$
\lim_{n\to\infty}\sup\{d(x,y):(x,y)\in\on{supp}(\xi_{k_n^{(1)}})\}=$$ $$=\liminf_{n\to\infty}\sup\{d(x,y):(x,y)\in\on{supp}(\xi_n)\}.
$$
Now since $(\on{supp}(\xi_{k_n^{(1)}}))$ is a sequence of compact sets in the compact space $\mathcal{K}(X^2)$ of all nonempty and compact subsets of $X^2$, we can find a subsequence $(\xi_{k^{(2)}_n})$ of $(\xi_{k_n^{(1)}})$ so that the sequence of supports $(\on{supp}(\xi_{k^{(2)}_n}))$ converges to some compact set $K\subset X^2$. Finally, since $I(X^2)$ is compact (w.r.t. the canonical topology $\tau_p$), we can find a subsequence $(\xi_{k_n})$ of $(\xi_{k_n^{(2)}})$ which converges to some $\xi\in I(X^2)$. Now we show that $\xi\in\Lambda_{\mu,\nu}$. Take any $\p\in C(X)$. Then we have
$$
\mu(\p)\leftarrow \mu_{n}(\p)=I(\pi_1)(\xi_{k_n})(\vp)=\xi_{k_n}(\p\circ\pi_1)\to \xi(\p\circ \pi_1)=I(\pi_1)(\xi)(\p)
$$
which shows that $I(\pi_1)(\xi)(\p)=\mu(\p)$. Since $\p$ was taken arbitrarily, we have $I(\pi_1)(\xi)=\mu$. Similarly we can show that $I(\pi_2)(\xi)=\nu$ and thus $\xi\in\Lambda_{\mu,\nu}$.\\
Now we observe that $\on{supp}(\xi)\subset K$. Suppose that it is not the case. Then there exists $s_0\in \on{supp}(\xi)\setminus K$. As $K$ is closed, we can find $\ve>0$ so that the closed ball $B(s_0,\ve)$ (w.r.t. some fixed metric $\rho$ on $X^2$; for later considerations, assume that $\rho$ is the maximum metric on $X^2$) is disjoint with $K$. By the Tietze extension theorem, we can find a continuous map $\p:X^2\to\R$ so that $\p(x)=0$ for $x\in K_{\ve/2}:=\{s\in X:\exists_{z\in K}\;\rho(s,z)\leq \frac{\ve}{2}\}$ and $\p(s_0)\geq 1-\eta(s_0)$, where $\eta$ is the density of $\xi$. Now we find $n_0\in\N$ so that for $n\geq n_0$, it holds $h(\on{supp}(\xi_{k_n}), K)<\frac{\ve}{2}$, where $h$ is the Hausdorrf metric. Then for $n\geq n_0$, we have $\on{supp}(\xi_{k_n})\subset K_{\ve/2}$ and
$$
\xi(\p)=\max\{\eta(x)+\p(x):x\in X\}\geq \eta(s_0)+\p(s_0)\geq 1
$$
and
$$
\xi_{k_n}(\p)=\max\{\eta_{k_n}(s)+\p(s):s\in X^2\}=$$ $$=\max\{\eta_{k_n}(s)+\p(s):s\in \on{supp}(\xi_{k_n})\}\leq 0
$$
where $\eta_{k_n}$ is the density of $\xi_{k_n}$. This leads to a contradiction with $\xi_{k_n}\to \xi$. Hence $\on{supp}(\xi)\subset K$.\\
Finally, take any $(x,y)\in \on{supp}(\xi)$ and take any $\ve>0$. Then find $n_0\in\N$ so that for $n\geq n_0$ we have $h(\on{supp}(\xi_{k_n}),K)<\ve$. As $(x,y)\in K$ and $\rho$ is the maximum metric on $X^2$, for any $n\geq n_0$ we can find $(x_n,y_n)\in \on{supp}(\xi_{k_n})$ so that $d(x,x_n),d(y,y_n)<\ve$. But then
$$d(x,y)\leq d(x,x_n)+d(x_n,y_n)+d(y_n,y)\leq \sup\{d(y,z):(y,z)\in\on{supp}(\xi_{k_n})\}+2\ve$$
and hence
$$
\sup\{d(x,y):(x,y)\in\on{supp}(\xi)\}\leq \lim_{n\to\infty}\sup\{d(x,y):(x,y)\in\on{supp}(\xi_{k_n})\}+2\ve.
$$
As $\ve>0$ was taken arbitrarily, and by the choice of $(k_n)$, we get
$$
\sup\{d(x,y):(x,y)\in\on{supp}(\xi)\}\leq \liminf_{n\to\infty}\sup\{d(x,y):(x,y)\in\on{supp}(\xi_{n})\}.
$$
\end{proof}
An immediate consequence of the Lemma \ref{newlemma1} is:
\begin{cor}\label{newcor1} The infimum occurring in the definition of metric $d_1$ is attained, that is, for any $\mu_1,\mu_2\in I(X)$, there exists $ \xi_0 \in \Lambda_{\mu_1,\mu_2}$ such that
$$
d_{1}\left(\mu_{1}, \mu_{2}\right):=\sup \{d(x, y):(x, y) \in \operatorname{supp} \xi_0\}
$$ where $\Lambda_{\mu_1,\mu_2}$ is the family of all idempotent measures $\xi\in I(X\times X)$ with $I\left(\pi_{i}\right)(\xi)=\mu_{i}, i=1,2$.
\end{cor}
Now we show that $d_1$ is complete:
\begin{lem}
The metric $d_1$ is complete.
\end{lem}
\begin{proof}
Take any $d_1$-Cauchy sequence $\mu_n$, $n\in\N$. Since $I(X)$ is compact w.r.t. the topology $\tau_p$, the sequence $(\mu_n)$ has a convergent subsequence. As Cauchy sequence is convergent iff some of its subsequence is convergent, WLOG we can assume that the sequence $(\mu_n)$ itself is convergent w.r.t. topology $\tau_p$. Let $\mu_0$ be its $\tau_p$-limit. Now take $n_1\in\N$ so that for $n\geq n_1$, it holds $d_1(\mu_{n_1},\mu_n)\leq\frac{1}{2}$. Now for every $n\geq n_1$, choose $\xi^1_n\in\Lambda_{\mu_{n_1},\mu_n}$ such that
 $$d_1(\mu_{n_1},\mu_n)=\sup\{d(x,y):(x,y)\in\on{supp}(\xi^1_n)\}.$$
 Such a choice is possible by Lemma \ref{newlemma1} (use it for $\mu_k=\mu:=\mu_{n_1},\;\nu_k=\nu:=\mu_n$, $k\in\N$, and appropriate sequence $(\xi_k)$ of measures from $\Lambda_{\mu_{n_1},\mu_n}$). Now, using Lemma \ref{newlemma1} again, but for $\mu_n=\mu:=\mu_{n_1}$, $\nu_n:=\mu_n$, $\nu:=\mu_0$ and $\xi_n$, $n\geq n_1$, we find a measure $\xi^1\in\Lambda_{\mu_{n_1},\mu_0}$ such that
$$
\sup\{d(x,y):(x,y)\in\on{supp}(\xi^1)\}\leq \liminf_{n\to\infty}\sup\{d(x,y):(x,y)\in\on{supp}(\xi^1_n)\}=$$ $$=\liminf_{n\to\infty}d_1(\mu_{n_1},\mu_n)\leq \frac{1}{2}.
$$
In particular, $d_1(\mu_{n_1},\mu_0)\leq\frac{1}{2}$. Using the same reasoning, we can find next values $n_1<n_2<n_3<...$ so that for every $k\in\N$, $d_1(\mu_{n_k},\mu_0)<\frac{1}{2^k}$. In particular, $(\mu_{n_k})$ is a convergent subsequence of $(\mu_n)$, and hence also the whole sequence $(\mu_n)$ converges. The result follows.
\end{proof}
Below we give a simple example that shows that $(I(X),d_1)$ need not be compact (in particular, it does not generate the topology $\tau_p$).
\begin{ex}{
Let $(X,d)$ be any compact space that has more than one element, and find distinct $x_0,y_0\in X$. For every $n\in\N$, let $\mu_n=\bigoplus_{x\in X}\lambda_n(x)\odot\delta_x$, where
$$\lambda_n(x)=\left\{\begin{array}{ccc}0&\mbox{if}&x=x_0\\-n&\mbox{if}&x=y_0\\-\infty&\mbox{if}&x\notin\{x_0,y_0\}\end{array}\right.$$
Now fix any $n\neq m$ and any $\xi=\bigoplus_{(x,y)\in X^2}\eta(x,y)\odot\delta_{(x,y)}\in\Lambda_{\mu_n,\mu_m}$. According to \cite[Proposition 3.1]{ZAI}, we have $\displaystyle \forall_{x\in X}\;\lambda_n(x)=\max\{\eta(x,y):y\in X\}$ and $\displaystyle\forall_{y\in X}\;\lambda_m(y)=\max\{\eta(x,y):x\in X\}.$ In particular, if $x\notin\{x_0,y_0\}$ or $y\notin\{x_0,y_0\}$, then $\eta(x,y)=-\infty$, and also:\\
$-n=\lambda_n(y_0)=\max\{\eta(y_0,x_0),\eta(y_0,y_0)\}$;\\
$-m=\lambda_m(y_0)=\max\{\eta(x_0,y_0),\eta(y_0,y_0)\}$.\\
Hence $\eta(x_0,y_0)>-\infty$ or $\eta(y_0,x_0)>-\infty$, so $(x_0,y_0)$ or $(y_0,x_0)$ belong to $\on{supp}(\xi)$. In particular, $d_1(\mu_n,\mu_m)\geq d(x_0,y_0)$ and $(\mu_n)$ has no $d_1$-convergent subsequence.
}\end{ex}

The following theorem gives an alternative version of the proof of the existence of invariant idempotent measure for Matkowski contractive max-plus normalized IFSs.

\begin{thm}\label{rho}
Assume that ${\mpS}=(X,(\phi_j)_{j=1}^L,(q_j)_{j=1}^L)$ is a Matkowski contractive max-plus {normalized IFS. Then the idempotent Markov operator $M_{\mathcal{S}}$ is Matkowski contractive w.r.t. $d_1$, with witness $\vp_\mS:=\max\{\vp_j:j=1,...,L\}$, where $\vp_j$s are witnesses for $\phi_j$s.

In particular, if $\mpS$ is Banach contractive, then $M_\mS$ is Banach contractive w.r.t. $d_1$ and $\on{Lip}(M_\mS)\leq\max\{\on{Lip}(\phi_j):j=1,...,L\}$.}
\end{thm}
\begin{proof}
Let $$\mu_1=\bigoplus_{x\in X}\lambda_1(x)\odot\delta_x,\;\;\;\mu_2=\bigoplus_{x\in X}\lambda_2(x)\odot\delta_x\in I(X).$$ Find
$$\xi=\bigoplus_{(x,y)\in X\times X}\eta(x,y)\odot\delta_{(x,y)}\in \Lambda_{\mu_1,\mu_2}
$$
so that
$$
d_1(\mu_1,\mu_2)=\sup\{d(x,y):(x,y)\in\on{supp}(\xi)\}.
$$
The existence of the measure $\xi$ follows from Corollary \ref{newcor1}.\\
By \cite[Proposition 3.1]{ZAI}, we have
\begin{equation}\label{abcd1}
\forall_{x\in X}\;\mu_1(x)=\max\{\eta(x,y):y\in X\}\;\mbox{and}\;\forall_{y\in X}\;\mu_2(y)=\max\{\eta(x,y):x\in X\}.
\end{equation}
Now let $\mu_i^\mS=\bigoplus_{s\in X}\lambda_i^{\mS}\odot\delta_s:=M_\mS(\mu_i)$ for $i=1,2$. Then by \cite[Lemma 5.5]{COS}, we have
\begin{equation}\label{abcd3}
\lambda_i^\mS(s)=\max\{q_j+\lambda_i(x):j=1,...,L,\;x\in \phi_j^{-1}(s)\}.
\end{equation}
Now define $$\overline{\mS}:=(X\times X,(\overline{\phi}_j)_{j=1}^L,(q_j)_{j=1}^L)$$ where $$\overline{\phi}_j(x,y):=(\phi_j(x),\phi_j(y)),\;\mbox{for}\;(x,y)\in X\times X.$$
Considering the maximum metric $d_m$ on $X\times X$, we have for every $j=1,...,L$ and $(x,y),(x',y')\in X\times X$,
$$
d_m(\overline{\phi}_j(x,y),\overline{\phi}_j(x',y'))=d_m((\phi_j(x),\phi_j(y)),(\phi_j(x'),\phi_j(y')))=$$ $$=
\max\{d(\phi_j(x),\phi_j(x')),d(\phi_j(y),\phi_j(y'))\}\leq$$ $$\leq \max\{\p(d(x,x')),\p(d(y,y'))\}\leq \p(\max\{d(x,x'),d(y,y')\})=$$ $$=\p(d_m((x,y),(x',y'))).
$$
Hence $\overline{\mS}$ consists of $\p$-contractions. Now let
$$
\xi_{\overline{\mS}}=\bigoplus_{(z,s)\in X\times X}\eta_{\overline{\mS}}(z,s)\odot\delta_{(z,s)}:=M_{\overline{\mS}}(\xi).
$$
Then by \cite[Lemma 5.5]{COS}, we have that the density
$$
\eta_{\overline{\mS}}(s,t)=\max\{q_j+\eta(x,y):j=1,...,L,\;(x,y)\in\overline{\phi}_j^{-1}(s,t)\}=$$ $$=
\max\{q_j+\eta(x,y):j=1,...,L,\;x\in\phi_j^{-1}(s),\;y\in\phi_j^{-1}(t)\}.
$$
Now we show that $\xi_{\overline{\mS}}\in\Lambda_{\mu_1^\mS,\mu_2^\mS}$. By \cite[Lemma 2.6]{COS}, the density of $I(\pi_1)(\xi_{\overline{\mS}})$ at $s\in X$ equals
$$
\max\{\eta_{\overline{\mS}}(x,y):(x,y)\in\pi_{1}^{-1}(z)\}=\max\{\eta_{\overline{\mS}}(s,t):t\in X\}=
$$
$$
=\max\{\max\{q_j+\eta(x,y):j=1,...,L,\;x\in\phi_j^{-1}(s),\;y\in\phi_j^{-1}(t)\}:t\in X\}=
$$
$$
=\max\{q_j+\eta(x,y):j=1,...,L,\;x\in\phi_j^{-1}(s),\;y\in X\}=
$$
$$
=\max\{q_j+\max\{\eta(x,y):y\in X\}:j=1,...,L,\;x\in\phi_j^{-1}(s)\}\overset{(\ref{abcd1})}{=}
$$
$$
=\max\{q_j+\lambda_1(x):j=1,...,L,\;x\in\phi_j^{-1}(s)\}\overset{(\ref{abcd3})}{=}\lambda^{\mS}_i(s)
$$
Hence $I(\pi_1)(\xi_{\overline{\mS}})=\mu_1^{\mS}$. Similarly we prove that $I(\pi_2)(\xi_{\overline{\mS}})=\mu_2^{\mS}$. Hence $\xi_{\overline{\mS}}\in\Lambda_{\mu_1^\mS,\mu_2^\mS}$. By \cite[Lemma 5.5]{COS}, it holds
$$
\on{supp}(\xi_{\overline{\mS}})=\on{supp}(M_{\overline{\mS}}(\xi))=\bigcup_{j=1}^L\overline{\phi}_j(\on{supp}(\xi)),
$$
so we have
$$
d_1(M_\mS(\mu_1),M_{\mS}(\mu_2))\leq \sup\{d(x,y):(x,y)\in \on{supp}(\xi_{\overline{\mS}})\}=
$$
$$
=\sup\{d(x,y):(x,y)\in \bigcup_{j=1}^L\overline{\phi}_j(\on{supp}(\xi))\}=
$$
$$
=\max\{\sup\{d(x,y):(x,y)\in \overline{\phi}_j(\on{supp}(\xi))\}:j=1,...,L\}=
$$
$$
=\max\{\sup\{d(\overline{\phi}_j(s,t)):(s,t)\in \on{supp}(\xi)\}:j=1,...,L\}=
$$
$$
=\max\{\sup\{d({\phi}_j(s),\phi_j(t)):(s,t)\in \on{supp}(\xi)\}:j=1,...,L\}\leq
$$
$$
\leq \max\{\sup\{\p(d(s,t)):(s,t)\in \on{supp}(\xi)\}:j=1,...,L\}=
$$
$$
=\sup\{\p(d(s,t)):(s,t)\in \on{supp}(\xi)\}\leq $$ $$\leq \p(\sup\{d(s,t):(s,t)\in \on{supp}(\xi)\})=\p(d_1(\mu_1,\mu_2)).
$$
All in all, $M_\mS$ is $\p$-contraction.
\end{proof}

\section{Contractivity w.r.t. a version of Bazylevych-Repov\v{s}-Zarichnyi's $\tilde{d}$ metric}
Again, throughout the rest of this section, we assume that $(X,d)$ is a compact metric space.\\
For each $a>0$ and $\mu,\nu\in I(X)$, define
$$
d_a(\mu,\nu)=\sup\{|\mu(\phi)-\nu(\phi)|:\phi\in\on{Lip}_a(X)\}
$$
where $\on{Lip}_a(X)$ is the family of maps $\phi:X\to\R$ with $\on{Lip}(\phi)\leq a$. In \cite[Theorem 4.1]{Rep10}, the authors prove that $d_a$  are continuous pseudometrics for each $a\in\N$, and that $\tilde{d}$ defined by
$$
\tilde{d}(\mu,\nu):=\sum_{n=1}^\infty\frac{d_n(\mu,\nu)}{n\cdot 2^n},\;\;\mu,\nu\in I(X)
$$
is a metric on $I(X)$ that generates the canonical topology $\tau_p$. We will show that idempotent Markov operator for Banach contractive max-plus normalized IFS is a Banach contraction w.r.t. some natural modification of $\tilde{d}$.\\
For $\alpha,q\in(0,1)$, define $\tilde{d}_{a,q}$ by
$$
\tilde{d}_{\alpha,q}(\mu,\nu):=\sum_{n\in\mathbb{Z}}\frac{q^{|n|}}{\alpha^n}d_{\alpha^n}(\mu,\nu)
$$
In a similar way as in \cite{Rep10} we can show that $\tilde{d}_{\alpha,q}$ is a metric that generates the topology $\tau_p$ (we just have to observe that $\tilde{d}_{\alpha,q}$ is well defined and that the family $d_{\alpha^n}$, $n\in\Z$ is a family of continuous pseudometrics that separates points). Note that $(X,d_{\alpha^n})$ is compact as it generates the compact topology $\tau_p$. In particular,  $(X,d_{\alpha^n})$ is complete.

The following theorem give an alternative version of the proof of the existence of invariant idempotent measure for Banach contractive max-plus normalized IFSs.

\begin{thm}
Assume that ${\mpS}=(X,(\phi_j)_{j=1}^L,(q_j)_{j=1}^L)$ is a Banach contractive max-plus {normalized IFS. Let
$$
\alpha:=\max\{\on{Lip}(\phi_j):j=1,...,L\}
$$
and choose $q\in(\alpha,1)$. Then $M_\mS$ is Banach contraction w.r.t. $d_{\alpha,q}$. More precisely,
$$
\on{Lip}(M_\mS)\leq \frac{\alpha}{q}.
$$}
\end{thm}
\begin{proof}
Take any $\mu=\bigoplus_{x\in X}\lambda(x)\odot\delta_x\in I(X)$ and a continuous map $\p:X\to\R$. By \cite[Lemma 5.5]{COS}, we have
$$
M_\mS(\mu)(\p)=\max\{\lambda_\mS(y)+\p(y):y\in X\}=$$ $$=\max\{\max\{q_j+\lambda(x):j=1,..,L,\;x\in\phi_j^{-1}(y)\}:y\in X\}=
$$
$$
=\max\{q_j+\lambda(x)+\p(\phi_j(x)):j=1,...,L,\;x\in X\}=$$ $$=\max\{\lambda(x)+\max\{q_j+\p\circ\phi_j(x):j=1,...,L\}:x\in X\}=\mu(\p_\mS)
$$
for $\p_\mS:=\max\{q_j+\p\circ\phi_j:j=1,...,L\}$. It is easy to see that if $\p$ is Lipschitz, then so is $\p_\mS$ and $\on{Lip}(\p_\mS)\leq \alpha\cdot\on{Lip}(\p)$. Hence, choosing $\mu_1,\mu_2\in I(X)$ and $\p\in \on{Lip}_{\alpha^n}(X)$, we have that $\p_\mS\in\on{Lip}_{\alpha^{n+1}}(X)$ and thus
$$
|M_\mS(\mu_1)(\p)-M_\mS(\mu_2)(\p)|=|\mu_1(\p_\mS)-\mu_2(\p_\mS)|\leq d_{\alpha^{n+1}}(\mu_1,\mu_2).
$$
Since $\p$ was chosen arbitrarily, we have
$$
d_{\alpha^n}(M_\mS(\mu_1),M_\mS(\mu_2))\leq d_{\alpha^{n+1}}(\mu_1,\mu_2)
$$
and
$$
\tilde{d}_{\alpha,q}(M_\mS(\mu_1),M_\mS(\mu_2))=\sum_{n\in\Z}\frac{q^{|n|}}{\alpha^n}d_{\alpha^n}(M_\mS(\mu_1),M_\mS(\mu_2))\leq
$$
$$
\leq \sum_{n\in\Z}\frac{q^{|n|}}{\alpha^{n}}d_{\alpha^{n+1}}(\mu_1,\mu_2)=
\sum_{n\in\Z}\frac{q^{|n|}\cdot \alpha}{q^{|n+1|}}\cdot\frac{q^{|n+1|}}{\alpha^{n+1}}d_{\alpha^{n+1}}(\mu_1,\mu_2)\leq
$$
$$
\leq \frac{\alpha}{q}\sum_{n\in\Z}\frac{q^{|n|}}{\alpha^{n}}d_{\alpha^{n}}(\mu_1,\mu_2)=
\frac{\alpha}{q}\tilde{d}_{\alpha,q}(\mu_1,\mu_2).
$$

\end{proof}$\;$\\


\begin{thebibliography}{1}
\bibitem[Aki99]{Aki99}  M. Akian, \emph{Densities of idempotent measures and large deviations}. Trans. Amer. Math. Soc. \textbf{351} (1999), 4515--4543.

\bibitem[Bar88]{BAR88}
M.~F. Barnsley, {\em Fractals everywhere}. Academic Press, 1988.

\bibitem[BRZ10]{Rep10} L. Bazylevych, D. Repov\v{s}, M. Zarichnyi,
 {\em Spaces of idempotent measures of compact metric spaces.}
 Topology Appl. \textbf{157} (2010), no. 1, 136--144.

\bibitem[COS21]{COS} R. D. da Cunha, E. R. Oliveira and F. Strobin, \emph{Fuzzy-set approach to invariant idempotent measures}, submitted, arXiv: 2109.13040.

\bibitem[DD99]{MoDo99} P. Del Moral, M. Doisy, \emph{Maslov idempotent probability calculus. I.} Theory Probab. Appl. \textbf{43} (1999), no. 4, 562--576.

\bibitem[Hut81]{HUT}
J. Hutchinson, \emph{Fractals and self-similarity}.
{Indiana Univ. Math. J.} \textbf{30} (1981), 713--747.

\bibitem[JJ07]{JJ07} J. Jachymski and I. J\'o\'zwik,
\emph{Nonlinear contractive conditions: a comparison and related problems.} Banach Center Publ. \textbf{77} (2007), 123--146.

\bibitem[KM88]{Kol88} V. N. Kolokoltsov and V. P. Maslov, \emph{The general form of the endomorphisms in the space of continuous functions with values in a numerical semiring}. Sov. Math. Dokl. \textbf{36} (1988), 55--59.

\bibitem[KM97]{Kol97} V. N. Kolokoltsov and V. P. Maslov, \emph{Idempotent analysis and its applications}, Kluwer Publishing House, 1997.

\bibitem[LMS02]{Lit02} G. L. Litvinov, V. P. Maslov and G. B. Shpiz, \emph{Idempotent (asymptotic) analysis and the representation theory}. Asymptotic combinatorics with application to mathematical physics (St. Petersburg, 2001), 267--278, NATO Sci. Ser. II Math. Phys. Chem., 77, Kluwer Acad. Publ., Dordrecht, 2002.

\bibitem[Lit07]{Lit07} G. L. Litvinov, \emph{Maslov dequantization, idempotent and tropical mathematics: A brief introduction}. J. Math. Sci. (N.Y) 140 (2007), no. 3, 426--444.

\bibitem[MZ14]{MZ2} N.  Mazurenko, M. Zarichnyi, \emph{Idempotent ultrametric fractals}. Visnyk of the Lviv Univ. Series Mech. Math. 79 (2014), 111--118.


\bibitem[MZ18]{MZ} N. Mazurenko, M. Zarichnyi, \emph{Invariant idempotent measures,}
Carpathian Math. Publ. {10} (2018), no. 1, 172--178.

\bibitem[Zai20]{ZAI}
{A. A. Zaitov, \newblock {\em  On a metric of the space of idempotent probability measures.}
\newblock Appl. Gen. Topol. \textbf{21} (2020), no. 1, 35--51.}

\bibitem[Zar10]{Zar} M. M Zarichnyi,
 {\em Spaces and maps of idempotent measures, }
   Izv. Math. \textbf{74} (2010), no. 3, 481--499.

\end{thebibliography}
\end{document}